\documentclass[11pt,leqno]{amsart}
\usepackage{amsmath,amssymb,amsthm}
\usepackage{hyperref}

\usepackage{tikz-cd}

\usepackage{hyperref}

\usepackage{bbm}

\newcommand{\R}{\mathbb{R}}

\DeclareMathOperator{\diam}{diam\,}
\DeclareMathOperator{\co}{co}

\newcommand{\lspan}{\operatorname{span}}

\renewcommand{\geq}{\geqslant}
\renewcommand{\leq}{\leqslant}

\newcommand{\norm}[1]{\left\Vert#1\right\Vert}

\newcommand{\cco}{\overline{\operatorname{co}}}

\newtheorem{theorem}{Theorem}[section]
\newtheorem{lemma}[theorem]{Lemma}

\newtheorem{proposition}[theorem]{Proposition}

\theoremstyle{definition}
\newtheorem{definition}[theorem]{Definition}

\theoremstyle{remark}
\newtheorem{remark}[theorem]{Remark}

\numberwithin{equation}{section}
\newcommand{\abs}[1]{\left\lvert#1\right\rvert}

\def\fnote#1{\footnote}

\def\natu{{\mathbb N}}

\def\ignora#1{}
\def\n3#1{\left\vert  \! \left\vert \! \left\vert \, #1 \, \right\vert \!
  \right\vert \! \right\vert }

\renewcommand{\leq}{\le}

\usepackage{accents}
\usepackage[most]{tcolorbox}
\usepackage{comment}

\let\emptyset\varnothing

\newcommand{\N}{\mathbb{N}}

\newcommand{\eps}{\varepsilon}
\newcommand{\conv}{\text{co}}

\begin{document}

\author{ Gin\'es L\'opez-P\'erez }\address{Universidad de Granada, Facultad de Ciencias. Departamento de An\'{a}lisis Matem\'{a}tico, 18071-Granada
(Spain)} \email{ glopezp@ugr.es}

\author{ Esteban Martínez Vañó }\address{Universidad de Granada, Facultad de Ciencias. Departamento de An\'{a}lisis Matem\'{a}tico, 18071-Granada
(Spain)} \email{ emv@ugr.es}

\author{ Abraham Rueda Zoca }\address{Universidad de Granada, Facultad de Ciencias. Departamento de An\'{a}lisis Matem\'{a}tico, 18071-Granada
(Spain)} \email{ abrahamrueda@ugr.es}
\urladdr{\url{https://arzenglish.wordpress.com}}



\title{Big weak open radius versus big slice diameter}

\begin{abstract}
We show that there exists a Banach space in which every non-empty weakly open subset of its unit ball has radius one, the maximum possible value, but the infimum of the diameter of its slices is exactly one, so extremely far from its maximum. In fact, we show that there is a wide class of non-isomorphic Banach spaces satisfying this extreme difference between the behaviour of the radius and the diameter of non-empty weakly open subsets.
\end{abstract}

\maketitle

\markboth{G. L\'OPEZ-P\'EREZ, E. MART\'INEZ VAÑ\'O, A. RUEDA ZOCA}{BIG WEAK OPEN RADIUS VERSUS BIG SLICE DIAMETER}

\section{Introduction}

The study of the geometry behind slices, weakly open sets and convex combinations of slices has been central in several studies in Functional Analysis. To mention an important result in that line, the existence of slices (respectively non-empty weakly open subsets, convex combinations of slices) of arbitrarily small diameter or arbitrarily small radius in every closed, convex and bounded subset of the space yields a characterisation of the Radon-Nikodym property (RNP) (respectively of the convex point of continuity property (CPCP), strong regularity (SR)). 

In the opposite extreme, we can find that the study of big slices, weakly open subsets and convex combinations of slices has been carried out through the diameter two properties \cite{ahntt16,almt21,aln13,blr15eje1,blr15eje2} and the r-BSP \cite{hllnr20,ivakhno06,rz23}.

Recall that a Banach space $X$ has the \textit{slice-diameter two property} (slice-D2P), respectively the \textit{diameter two property} (D2P), the \textit{strong diameter two property} (SD2P), if every slice of the unit ball, respectively every non-empty relatively weakly open subset, every convex combination of slices of the unit ball, has diameter exactly 2. Moreover, $X$ has the \textit{r-big slice property} (r-BSP) if every slice $S$ of the unit ball $B_X$ has radius exactly 1, in other words, if given any $x\in X$ and any $\varepsilon>0$ there exists a $y\in S$ such that $\Vert x-y\Vert \geq1-\varepsilon$. 

Concerning the diameter two properties, even though its study has been implicit in connection with other geometric properties of Banach spaces since the eighties (for instance in the study of rough norms \cite{jozi78}, octahedral norms \cite{deville88} and the Daugavet property \cite{kssw01}), the starting point of their development as an independent research topic can be traced back to \cite{nw01}. It is immediate that the D2P implies the slice-D2P, whereas the reverse implication fails in an extreme way \cite[Theorem 2.4]{blr15eje1} (another extreme counterexample was latter given in \cite[p. 395]{almt21}). In a similar way, it is known that the SD2P implies the D2P (c.f. e.g. \cite[p. 440]{aln13}), whereas once again the converse implication fails in an extreme way \cite[Theorem 2.5]{blr15eje2} (once again, another extreme counterexample was latter provided in \cite[Theorem 2.12]{ahntt16}).

Concerning the r-BSP, this property was introduced in \cite{ivakhno06} together with the slice-D2P (under the name of \textit{d-big slice property}), making a systematic study of both properties. It is proved there that the slice-D2P implies the r-BSP, leaving the converse as an open question in \cite[p. 96]{ivakhno06}. A negative answer was provided in \cite[Theorem 3.7]{hllnr20}, where it was proved that the space $X=(JT_\infty)_*$ has the r-BSP but the infimum of the diameter of slices of the unit ball is $\sqrt{2}$. Such counterexample was put further in \cite[Theorem 1.1]{rz23}, where there is an example of a Banach space $X$ with the r-BSP, the CPCP and such that the infimum of the diameter of slices of the unit ball equals $1$. This example shows that the r-BSP and the slice-D2P are as far as they can be. 

The above difference between the r-BSP and the slice-D2P has another interesting implication from an isomorphic point of view. It has been conjectured that every Banach space failing the RNP admits an equivalent renorming with the slice-D2P (c.f. e.g. \cite[p. 553]{bl06}). The above difference between the r-BSP and the slice-D2P implies, however, that the following conjecture does make more sense: every Banach space $X$ failing the RNP admits an equivalent renorming with the r-BSP. Observe that it was proved in \cite[Theorem 2]{ivakhno06} that every Banach space failing the RNP admits, for every $\varepsilon>0$, an equivalent renorming such that every slice of the new unit ball has radius at least $1-\varepsilon$. 

In this paper we consider the question whether the difference between the r-BSP and the slice-D2P could be translated from slices to weakly open subsets. Namely, we introduce in Definition~\ref{def:r-BWOP} the \textit{r-big weak open property} (r-BWOP), which is clearly implied by the D2P. We face the question whether the r-BWOP actually implies the D2P. Let us point out that we will not consider here the similar version for convex combinations of slices since we produce nothing new but the SD2P (see Remark~\ref{remark:versionccslices}).

Coming back to our main question, the main result of the paper is the following.

\begin{theorem}\label{maintheorem}
There exists a Banach space $X$ with the \emph{r-BWOP} and such that
$$\inf\left\{\diam(S): S\subseteq B_X\mbox{ \emph{is a slice}} \right\}=1.$$
\end{theorem}
An inspection in the proof reveals that not only we prove the existence of such space $X$ but also we find a large class of non-isomorphic examples (see Remark~\ref{remark:conseqmaintheor}).

After fixing a bit of notation, we will split the content of the paper in two sections. In Section~\ref{section:psumas} we will introduce the formal definition of the r-BWOP together with some observations about it. Then we will describe the behaviour of the r-BWOP by classical $\ell_p$-sums in Proposition~\ref{infsuma} (case $p=\infty$) and in Proposition~\ref{psuma} (case $1\leq p<\infty$). Apart from completeness, all the above mentioned stability results will be used in the proof of Theorem~\ref{maintheorem}, which will be developed in Section~\ref{section:demoprincipal}. The (involved) proof of such result will be presented in several steps: first of all, we prove in Lemma~\ref{renormacion} that $c_0\oplus_\infty\mathbb R$ has a renorming with the r-BWOP and such that the infimum of the diameter of slices of the unit ball is at most $1+\delta$. The proof of this result will make use of ideas from \cite[Theorem 2.5]{blr15eje2}. We will then obtain a generalisation in Theorem~\ref{c0compl}, where we extend Lemma~\ref{renormacion} from $c_0\oplus_\infty\mathbb R$ to any Banach space $X$ containing a complemented copy of $c_0$. Finally we present the proof of Theorem~\ref{maintheorem} and conclude with some remarks and open questions.

\textbf{Terminology:} We only consider real Banach spaces. The closed unit ball of a Banach space $X$ is denoted
by $B_X$ and its unit sphere by $S_X$. The dual space of $X$ is denoted by $X^\ast$ and the bidual by $X^{\ast\ast}$.

By a \textit{slice} of $B_X$ we mean a set of the form
\begin{equation*}
S(B_X, x^*,\alpha) :=
\{
x \in B_X : x^*(x) > 1 - \alpha
\},
\end{equation*}
where $x^* \in S_{X^*}$ and $\alpha > 0$. 

We write ${\rm co}(D)$ (resp., $\overline{{\rm co}}(D)$) to denote the convex hull (resp., closed convex hull) of~$D$. Moreover, we will denote by
$$U(x_0;f_1,\ldots, f_n;\varepsilon):= \bigcap_{i=1}^n\{x\in X: \vert f_i(x-x_0)\vert<\varepsilon\},$$
where $f_1,\ldots, f_n\in X^*, x_0\in X$ and $\varepsilon>0$.

\section{\texorpdfstring{\textnormal{r}}{r}-big weak open property. Stability properties}\label{section:psumas}

We begin the section with the definition of the radius in a Banach space:

\begin{definition}
Let $X$ be a Banach space and $C \subset X$ be a bounded subset of $X$. We define the \textit{radius} of $C$ as
$$r(C) = \inf \left\{ t>0: \exists x \in X \text{ such that } C \subset x + tB_X \right\}.$$
\end{definition}

We will use the following easy characterization along the text:

\begin{lemma}\label{charradio}
Let $X$ be  Banach space, $C \subset X$ a bounded subset of $X$ and $r \geq 0$. Then, $r(C) \geq r$ if, and only if, for any $x \in X$ and $\eps > 0$, there exists some $c \in C$ with $\norm{x-c} \geq r - \eps$. 
\end{lemma}

As we mentioned in the introduction, we are interested in the introduction of the following generalization of the r-BSP:

\begin{definition}\label{def:r-BWOP}
Let $X$ be a Banach space. We say that $X$ has the \textit{r-big weak open property} (r-BWOP) if every non-empty weakly open subset of its unit ball has radius one.
\end{definition}

Observe that given any bounded subset $C \subset X$ it is obvious that
$$r(C) \leq \diam(C) \leq 2r(C),$$
so any space with the slice-D2P (D2P) will clearly have the r-BSP (r-BWOP), and in the opposite side, if there exist slices (weakly open subsets) of the unit ball with arbitrarily small diameter, then there are also slices (weakly open subsets) with arbitrarily small radius. This allows us to differentiate the r-BSP from the r-BWOP, as there exist Banach spaces with the slice-D2P whose unit ball contains non-empty relatively weakly open subsets of arbitrarily small diameter (see \cite[Theorem 2.4]{blr15eje1} or \cite[Theorem 3.1]{rz23}, where there are examples with the r-BSP, the CPCP and failing the slice-D2P). Furthermore, Theorems \ref{maintheorem} and \ref{c0compl} prove that the r-BWOP is a different property that the D2P in the most extreme possible way.

\begin{remark}\label{remark:versionccslices}
By the previous comments one could be tempted to define an analogous property to the SD2P for the radius, but it turns out that this one would be equivalent to the SD2P.

In fact, if every convex combination $C$ of slices of the unit ball of a Banach space $X$ has radius one, then by Lemma \ref{charradio} (with $x = 0$) we obtain that for any $\eps > 0$ there is a $c \in C$ with $\norm{c} \geq 1-\eps$, so by \cite[Theorem 3.1]{lmr} we conclude that $X$ has the SD2P.
\end{remark}

\begin{remark}\label{remark:conedaugaindex}
Recall that in \cite[Section 4]{rz18} the following Daugavet index of thickness is introduced:
\begin{equation*}
  \mathcal{T}(X) = \inf\left\lbrace r>0 \;\middle|\;
  \begin{tabular}{@{}l@{}}
    \text{ there exist $x\in S_X$ and a relatively weakly }\\ \text{ open $W$ in $B_X$ such that} $\emptyset\neq W \subset B(x,r)$
   \end{tabular}
  \right\rbrace.
\end{equation*}
The above index of thickness was introduced (in an equivalent formulation) to provide a quantitative measure of how far a Banach space $X$ is from having the \textit{Daugavet property}.

Observe that if a Banach space $X$ has the r-BWOP then $\mathcal T(X)\geq 1$. The converse does not hold true. Indeed, in \cite[Theorem 1.2]{rz23} it is proved that, given $\varepsilon>0$, there exists a Banach space $X$ which fails the r-BSP and such that $\mathcal T(X)\geq 2-\varepsilon$.
\end{remark}

We prove now a couple of results about the stability of the r-BWOP by $p$-sums which will allow us to prove Theorem \ref{maintheorem}. Observe that these results are analogous to the ones for diameter two properties. Let us start with the case of $p=\infty$.

\begin{proposition}\label{infsuma}
Let $X$ and $Y$ be Banach spaces. Then, the space $X \bigoplus_\infty Y$ has the \emph{r-BWOP} if, and only if, one of $X$ or $Y$ has the \emph{r-BWOP}.
\end{proposition}

\begin{proof}
Let us assume that $X$ has the r-BWOP (the case for $Y$ is analogous) and let $U$ be a non-empty weakly open subset of $B_Z$, where $Z = X \bigoplus_\infty Y$. Without lost of generality we can consider that
$$U = \bigcap_{i=1}^n\{(x,y) \in B_Z: |f_i(x-x_0) + g_i(y-y_0)| < \eps\}$$
for some $f_1, \cdots, f_n \in X^*$, $g_1, \cdots, g_n \in Y^*$ and $(x_0, y_0) \in B_Z$. Then,
\begin{align*}
    & U_1 = \bigcap_{i=1}^n\{x \in B_X: |f_i(x-x_0)| < \eps/2\},\\
    & U_2 = \bigcap_{i=1}^n\{y \in B_Y: |g_i(x-y_0)| < \eps/2\}
\end{align*}
are non-empty weakly open subsets of $B_X$ and $B_Y$ respectively and clearly $U_1 \times U_2 \subset U$. Given any $\delta > 0$ and $(x,y) \in X \times Y$, as $X$ has the r-BWOP, Lemma \ref{charradio} guarantees the existence of some $z \in U_1$ with $\norm{x-z} > 1-\delta$, so we have that
$$\norm{(x,y) - (z,y_0)}_\infty \geq \norm{x - z} \geq 1-\delta.$$
As $(z,y_0) \in U_1 \times U_2 \subset U$ by Lemma \ref{charradio} we obtain that $r(U) \geq 1$, so that $r(U) = 1$.

Conversely, if we suppose that $X$ and $Y$ do not have the r-BWOP, then there are some weakly open subsets $U_X, U_Y$  of $X$ and $Y$ respectively such that $U_X \cap B_X, U_Y \cap B_Y \not = \emptyset$ and
$$r(U_X \cap B_X), r(U_Y \cap B_Y) < 1.$$
This implies the existence of some $0 < t  <1$ and $x \in X, y \in Y$ such that
\begin{align*}
    & U_X \cap B_X \subset x + tB_X, \\
    & U_Y \cap B_Y \subset y + tB_Y.
\end{align*}
Defining $U = (U_X \times U_Y) \cap B_{X \bigoplus_\infty Y}$ it is obviously a non-empty weakly open subset of the unit ball of $X \bigoplus_\infty Y$ that satisfies
$$U \subset (x,y) + t B_{X \bigoplus_\infty Y}.$$
The previous inclusion proves that $r(U) \leq t < 1$, so $X \bigoplus_\infty Y$ does not have the r-BWOP.
\end{proof}

For $1 \leq p < \infty$ we can extend the result to arbitrary sums:

\begin{proposition}\label{psuma}
Let $\{X_i\}_{i \in I}$ be a family of Banach spaces and $1 \leq p < \infty$. If $\ell_p(I, X_i)$ denotes the $\ell_p$-sum of the family Banach spaces $\{X_i\}_{i \in I}$, then the following are equivalent:
\begin{enumerate}
    \item Every $X_i$ has the \emph{r-BWOP}.
    \item The space $\ell_p(I, X_i)$ has the \emph{r-BWOP}.
\end{enumerate}
\end{proposition}

\begin{proof}
Along the proof we will denote $X = \ell_p(I,X_i)$.

(1) $\Rightarrow$ (2) We will use Lemma \ref{charradio} to prove that $X$ has the r-BWOP. So given any non-empty weakly open subset $W$ of $B_X$, $x \in X$ and $\eps > 0$ we have to prove that there exists some $w \in W$ such that
\begin{align}\label{psumaconcl}
\norm{w - x}_p \geq 1 - \eps.
\end{align}
Furthermore, taking into account that the slices of $B_X$ form a subbase for the weak topology in $B_X$ and using \cite[Lemma 1.4]{ivakhno04} we can assume that
$$W = \bigcap_{j=1}^n S(B_X, \phi_j, \alpha)$$
for some $\phi_j \in S_{X^*}$ and $0 < \alpha < \min\{1, \eps\}$. By the duality of $\ell_p$ spaces we also know that for every $j = 1, \cdots, n$ there exists some $f^j \in \ell_q(I, X_i^*)$ with norm one, where $q$ is the Hölder conjugate of $p$ (that is, it satisfies $\frac{1}{p}+\frac{1}{q}=1$), such that
$$\phi_j = \sum_{i \in I} f_i^j \circ \pi_i$$
where $\pi_i$ is the $i$-th projection.

Take any $z \in W$, that is, satisfying for any $j= 1, \cdots, n$
$$\phi_j(z) = \sum_{i \in I} f_i^j(z_i) > 1- \alpha.$$
Observe that given $0 < 2 \eta < \min \phi_j(z) - (1 - \alpha)$ and $0 < \Delta < (1-\alpha +2\eta)^p - (1-\alpha + \eta)^p$ there exists some finite subset $F$ of $I$ such that for any $j=1, \cdots, n$ we have that
\begin{align}\label{psuma->1}
\sum_{i \in F} f_i^j(z_i) > \sum_{i \in I} f_i^j(z_i) - \eta = \phi_j(z) - \eta > 1- \alpha + \eta.
\end{align}
and
\begin{align}\label{psuma->arreglo1}
    \sum_{i \in F} \norm{z_i}^p & > \sum_{i \in I} \norm{z_i}^p - \Delta = \norm{z}_p^p - \Delta \geq \abs{\phi_j(z)}^p - \Delta > \\
    & > (1-\alpha + 2\eta)^p - \Delta > (1-\alpha + \eta)^p.\nonumber
\end{align}
In the previous inequalities we can change $F$ for $G = \{i \in F: \norm{z_i} \not = 0\}$, so without lost of generality we will assume that $\norm{z_i} \not = 0$ for any $i \in F$.

We can now define for every $i \in F$ the following non-empty weakly open subset of $B_{X_i}$
$$W_i = U\left( \dfrac{z_i}{\norm{z_i}}; f_i^1, \cdots, f_i^n; \dfrac{\delta}{\norm{z_i}} \right) \cap B_{X_i}$$
with $0 < \delta < \eta/\abs{F}$ ($\vert F\vert$ stands for the cardinality of $F$) and, as $X_i$ has the r-BWOP, by Lemma \ref{charradio} there exists some $y_i \in W_i$ with 
\begin{align}\label{psuma->3}
\norm{x_i - y_i} \geq 1 - \theta
\end{align}
for $0 < \theta < \frac{\eta}{1- \alpha + \eta} \min_{i \in F} \norm{z_i}$.

Defining $w \in X$ such that $w_i = \norm{z_i} y_i$ if $i \in F$ and $w_i = 0$ otherwise, let us see that it is the desired point:
\begin{itemize}
    \item $w \in W$ because, as every $y_i \in B_{X_i}$ and $z \in B_X$, we get
    $$\norm{w}_p^p = \sum_{i \in F} \norm{z_i}^p \norm{y_i}^p \leq \sum_{i \in F} \norm{z_i}^p \leq \norm{z}^p \leq 1.$$
    Moreover, for every $j= 1, \cdots,n$, by taking into account (\ref{psuma->1}) and that $y_i \in W_i$ for any $i \in F$, we obtain that
    \begin{align*}
        \phi_j(w) & = \sum_{i \in I} f_i^j(w_i) = \sum_{i \in F} \norm{z_i} f_i^j(y_i) > \sum_{i \in F} \left(f_i^j(z_i) - \delta\right) = \\
        & =\sum_{i \in F} f_i^j(z_i) - \abs{F}\delta > 1-\alpha + \eta - \abs{F}\delta > 1- \alpha.
    \end{align*}
    \item $w$ satisfies (\ref{psumaconcl}) because by (\ref{psuma->arreglo1}) and (\ref{psuma->3})
    \begin{align*}
        \norm{w - x}_p^p & = \sum_{i \in I} \norm{w_i - x_i}^p \geq \sum_{i \in F} \norm{\norm{z_i}y_i - x_i}^p \geq \\
        & \geq \sum_{i \in F} \left( \norm{x_i - y_i} - \abs{1 - \norm{z_i}} \norm{y_i} \right)^p \geq \\
        & \mathop{\geq}\limits^{\mbox{\tiny{(\ref{psuma->3})}}}  \sum_{i \in F} \left( 1-\theta - (1 - \norm{z_i}) \right)^p = \sum_{i \in F} \left( \norm{z_i} - \theta \right)^p = \\
        & = \sum_{i \in F} \norm{z_i}^p \left( 1 - \dfrac{\theta}{\norm{z_i}}\right)^p >\\
        & > \left ( 1- \dfrac{\eta}{1-\alpha + \eta} \right)^p \sum_{i \in F} \norm{z_i}^p \geq \\
        & \mathop{\geq}\limits^{\mbox{\tiny{(\ref{psuma->arreglo1})}}} \left ( 1- \dfrac{\eta}{1-\alpha + \eta} \right)^p (1-\alpha + \eta)^p = \\
        & = (1-\alpha)^p > (1-\eps)^p.
    \end{align*}
\end{itemize}
This concludes this implication.

(2) $\Rightarrow$ (1). As before, by Lemma \ref{charradio} it is enough to fix $k \in I$, a non-empty weakly open subset $W_k$ of $B_{X_k}$, $\eps > 0$ and a point $x \in X_k$, and prove that there is some point $y \in W_k$ such that $\norm{x - y} \geq 1 - \eps$.

We can assume, without lost of generality, that
$$W_k = \bigcap_{j=1}^n S(B_{X_{k}}, f_j, \alpha)$$
for some $f_1, \cdots, f_n \in S_{X_k^*}$ and $0 < \alpha < \eps$ and define the non-empty weakly open subset of $B_X$
$$W = \bigcap_{j=1}^n S(B_{X}, f_j \circ \pi_k, \alpha).$$
As $X$ has the r-BWOP, by Lemma \ref{charradio} we obtain a $w \in W$ such that
\begin{align}\label{radiogrande}
\norm{w - z}_p \geq 1 - \eta
\end{align}
where $z \in X$ is defined as $z_k = x$ and $z_i = 0$ otherwise, and 
$$0 < \eta < 1-\left( (1-\eps)^p + 1 - (1 -\alpha)^p \right)^{1/p}.$$
By the definition of $W$ it is obvious that $w_k \in W_k$, so
$$\norm{w_k} \geq f_1(w_k) > 1- \alpha,$$
and
\begin{align}\label{psuma<-1}
\sum_{i \in I\setminus\{k\}} \norm{w_i}^p = \norm{w}_p^p - \norm{w_k}^p < 1-(1-\alpha)^p.
\end{align}
Using (\ref{radiogrande}) and (\ref{psuma<-1}) we obtain that
$$(1-\eta)^p \mathop{\leq}\limits^{\mbox{\tiny{(\ref{radiogrande})}}} \norm{w - z}_p^p = \norm{w_k - x}^p + \sum_{i \in I\setminus\{k\}} \norm{w_i}^p \mathop{<}\limits^{\mbox{\tiny{(\ref{psuma<-1})}}} \norm{w_k - x}^p + 1-(1-\alpha)^p,$$
so
$$\norm{w_k - x}^p > (1-\eta)^p + (1-\alpha)^p - 1 > (1-\eps)^p$$
and $w_k$ is the point we were looking for.
\end{proof}

\section{Proof of Theorem~\ref{maintheorem}}\label{section:demoprincipal}

The aim of the present section is to provide the proof of Theorem~\ref{maintheorem}, which will be made in several steps. The first one is to obtain an adequate renorming for $c_0\oplus_\infty\mathbb R$. In order to do so, we will get some ideas from \cite[Theorem 2.5]{blr15eje2}, for which we will need to recall a construction from Argyros, Odell, and Rosenthal \cite{aor}. 

Pick a strictly decreasing null sequence $\{\varepsilon_n\}$ in $\mathbb R^+$. We construct an increasing sequence of closed, bounded and convex subsets $\{K_n\}$ in the space $c_0$ and a sequence $\{g_n\}$ in $c_0$ as follows:

First define $K_1=\{e_1\}$, $g_1=e_1$ and $K_2=\conv(e_1, e_1+e_2)$. Choose $l_2>1$ and $g_2,\ldots ,g_{l_2}\in K_2$ an $\varepsilon_2$-net in $K_2$. Assume that $n\geq 2$ and that $m_n,\ l_n
,\ K_n$, and $\{g_1,\ldots ,g_{l_n}\}$ have been constructed, with
$K_n\subseteq B_{\lspan\{e_1,\ldots ,e_{m_n}\}}$ and $g_i\in K_n$ for every
$1\leq i\leq l_n$ such that $\{g_{l_{n-1}+1}, \cdots, g_{l_n}\}$ is an $\eps_n$-net in $K_n$. Define $K_{n+1}$ as $$K_{n+1}=\conv(K_n\cup
\{g_i+e_{m_n+i}\colon 1\leq i\leq l_n\}).$$ Consider $m_{n+1}=m_n+l_n$ and
choose $\{g_{l_{n}+1},\ldots ,g_{l_{n+1}}\}\subset K_{n+1}$ so that they form an $\varepsilon_{n+1}$-net in
$K_{n+1}$. Finally, we define $K_0=\overline{\cup_n K_n}$. 

It follows that $K_0$ is a non-empty closed, bounded and convex subset
of $c_0$ such that $x(n)\geq 0$ for every $n\in \natu$ and $\Vert
x\Vert_{\infty} = 1$ for every $x\in K_0$, so $\diam(K_0)\leq 1$. Furthermore, for a fixed $i$, we have from the construction that
$\{g_i+e_{m_n+i}\}_n$ is a sequence in $K_0$ (for $n$ large enough) which is weakly convergent to
$g_i$, and $\Vert (g_i-e_{m_n+i})-g_i\Vert=\Vert e_{m_n+i}\Vert=1$
holds for every $n$. Then $\diam(K_0)=1$.

We will freely use the set $K_0$ and the above construction in Lemma~\ref{renormacion}. Also observe that, from the construction, it follows that
$$K_0=\overline{\{g_i\colon i\in\natu\}}^{w}=\overline{\{g_i\colon i\in\natu\}}.$$
Finally one should also notice that, by the inductive construction, $g_i$ has finite support for every $i\in\mathbb N$.

Now we get the desired result for $c_0\oplus_\infty\mathbb R$.

\begin{lemma}\label{renormacion}
For every $\varepsilon \in (0,1)$, there exists an equivalent renorming $\abs{\cdot}_\eps$ of $c_0\bigoplus_\infty\R$ satisfying the \emph{r-BWOP} and such that
$$\inf\left\{ \diam_{\abs{\cdot}_\eps}(S): S \text{ \emph{slice of} } B_{\abs{\cdot}_\eps} \right\} \leq  \frac{1}{1-\varepsilon}.$$
\end{lemma}

\begin{proof}
Write $X = c_0\bigoplus_\infty\R$ and $Y=c_0\bigoplus_\infty \{0\}$.

Let $K_0 \subset c_0$ be the set previously defined and consider
$$K = \{(g, 1) \in X: g \in K_0\}.$$
Now we can define
$$B_\eps = \cco\left( K \cup -K \cup [(1-\eps)B_X  + \eps B_Y] \right),$$
which defines an equivalent norm $\abs{\cdot}_\eps$ on $X$ whose (closed) unit ball is $B_\eps$. Observe that
\begin{align}\label{1}
    (1- \eps) B_X \subset B_\eps \subset B_X,
\end{align}
so we infer
\begin{align}\label{2}
    \norm{(x,\alpha)}_\infty \leq \abs{(x,\alpha)}_\eps \leq \dfrac{1}{1-\eps}\norm{(x, \alpha)}_\infty,\, \forall x \in c_0, \alpha \in \R.
\end{align}

Let us begin by proving that $(X, \abs{\cdot}_\eps)$ enjoys the r-BWOP.

In order to do so, select a non-empty relatively weakly open subset $W$ of $X$ such that $W \cap B_\eps \not = \emptyset$, and let us prove that 
$$r(W \cap B_\eps) \geq 1.$$
To this end, by Lemma \ref{charradio} and the structure of $c_0$ it suffices to prove that given any $(x, \alpha) \in X$ where $x$ is finitely supported we can find a $z \in W \cap B_\eps$ satisfying
\begin{align}\label{3}
    \abs{(x,\alpha) - z}_\eps \geq 1.
\end{align}
Since the support of $x$ is finite, it is enough to find a sequence $\{(x_n, \alpha_n)\} \subset W \cap B_\eps$ such that, given $m \in \N$, there exists $n \in \N$ and $p \geq m$ with $|x_n(p)| = 1$. Indeed, as $x$ is finitely supported we can select $n_0 \in \N$ such that $x(n) = 0$ holds for every $n \geq n_0$, so taking $n \in \N$ and $p \geq n_0$ with $\abs{x_n(p)} = 1$ we get
$$\abs{(x, \alpha) - (x_n, \alpha_n)}_\eps \geq \norm{(x, \alpha) - (x_n, \alpha_n)}_\infty \geq \abs{x(p) - x_n(p)} = 1.$$
Consequently let us find such sequence.

Since $K$ and $C_\eps = (1-\eps)B_X + \eps B_Y$ are convex subsets satisfying $\frac{1}{2}(K-K) \subset C_\eps$, \cite[Lemma 2.4]{blr15eje2} implies
$$B_\eps = \cco(K \cup C_\eps) \cup \cco(-K \cup C_\eps).$$

By the above then either $W\cap  \cco(K \cup C_\eps)\neq \emptyset$ or $W\cap \cco(-K \cup C_\eps)\neq \emptyset$. We can thus assume, without lost of generality, that $W\cap  \cco(K \cup C_\eps)\neq \emptyset$ and, since $W$ is weakly open, select $w \in W \cap \co(K \cup C_\eps)$, so we can write
$$w = \lambda(y,1) + (1-\lambda)[(1-\eps)(x_0, \alpha_0) + \eps(x_1, 0)]$$
for suitable $\lambda \in [0,1]$, $y \in K_0$, $(x_0, \alpha_0) \in B_X$ and $x_1 \in B_{c_0}$. 

On the other hand, using once again that $W$ is weakly open there exists an $\eta>0$ such that $w + \eta B_X \subset W$ and, since
$$y \in K_0  = \overline{\{g_i: i \in \N\}},$$
we can find a $g_i$ such that
\begin{align}\label{4}
\norm{y - g_i}_\infty < \eta.
\end{align}
Moreover, since $x_0, x_1 \in c_0$ there are finitely supported elements $x_0', x_1' \in c_0$ verifying
\begin{align}
    & \norm{x_j'}_\infty \leq \norm{x_j}_\infty,\, j=0, 1 \label{5}\\
    & \norm{x_j - x_j'}_\infty < \eta,\, j = 0, 1. \label{6}
\end{align}
If we define
$$w' = \lambda(g_i,1) + (1-\lambda)[(1-\eps)(x_0', \alpha_0) + \eps(x_1', 0)],$$
(\ref{5}) and (\ref{6}) imply
$$\norm{w -w'}_\infty = \norm{\lambda(y-g_i) + (1-\lambda)[(1-\eps)(x_0 - x_0') + \eps (x_1 - x_1')]}_\infty < \eta,$$
so $w' \in W \cap \co(K \cup C_\eps)$.

By the definition of $K_0$ we can find a strictly increasing sequence $\{m_n\} \subset \N$ such that
\begin{align}\label{7}
g_i + e_{m_n+i} \in K_0, \, \forall n \in \N.
\end{align}
Furthermore, since $g_i, x_0', x_1'$ are finitely supported we can find $n_0 \in \N$ satisfying
\begin{align}\label{8}
g_i(n) = x_0'(n) = x_1'(n) = 0, \, \forall n \geq n_0.
\end{align}
Finally, define
$$w_n = \lambda(g_i + e_{m_n+i},1) + (1-\lambda)[(1-\eps)(x_0' + e_{m_n+i}, \alpha_0) + \eps(x_1' + e_{m_n+i}, 0)].$$
Clearly the sequence $\{w_n\}$ is weakly convergent to $w'$ and, since $w'$ belongs to the weakly open set $W$, we deduce the existence of some $n_1 \in \N$ verifying
\begin{align}\label{9}
w_n \in W, \, \forall n \geq n_1.
\end{align}
Taking $N=\max\{n_0, n_1\}$ we get that $\{w_n\}_{n=N}^\infty$ is the desired sequence as, by (\ref{7}), (\ref{8}) and (\ref{9}), it is contained in $W \cap \co(K \cup C_\eps)$ and by (\ref{8}) we get that, given $m \in \N$, we can select $n \geq N$ such that $m_n \geq m$, obtaining
$$\abs{\left[ \lambda(g_i + e_{m_n+i}) + (1-\lambda)[(1-\eps)(x_0' + e_{m_n+i}) + \eps(x_1' + e_{m_n+i})] \right](m_n + i)} = 1.$$

In order to prove the second assertion it is enough to show that, given $0 < t < 1$, the slice
$$S_t = \{(x, \alpha) \in B_\eps: \pi_2(x,\alpha) > 1- t\}$$
satisfies that
\begin{align}\label{10}
\diam_\eps(S_t) \leq \dfrac{1}{1-\eps} + \dfrac{3t}{\eps}.   
\end{align}
In order to do so, take $(x, \alpha), (y, \beta) \in S_t$. By a density argument we can assume that both elements belong to $S_t \cap \co(K \cup -K \cup C_\eps)$, so we can find $\lambda_1, \lambda_2, \lambda_3, \lambda_1', \lambda_2', \lambda_3' \geq 0$ with $\sum \lambda_i = \sum\lambda_i' = 1$, $a, a', b, b' \in K_0$, $(x_0, \alpha_0), (y_0, \beta_0) \in B_X$ and $x_1, y_1 \in B_{c_0}$ satisfying
\begin{align*}
& (x, \alpha) = \lambda_1(a,1) - \lambda_2(b,1) + \lambda_3 [(1-\eps)(x_0, \alpha_0) + \eps(x_1, 0)], \\
& (y, \beta) = \lambda_1'(a',1) - \lambda_2'(b',1) + \lambda_3' [(1-\eps)(y_0, \beta_0) + \eps(y_1, 0)].
\end{align*}
Since $(x,\alpha) \in S_t$ we get
$$1-t < \lambda_1 - \lambda_2 + \lambda_3(1-\eps)\alpha_0 \leq \lambda_1 - \lambda_2 + (1-\eps) \lambda_3 = 1 - 2 \lambda_2 - \eps\lambda_3,$$
so $2 \lambda_2 + \eps\lambda_3 < t$ and thus
\begin{align}\label{17}
\lambda_2 + \lambda_3 = \dfrac{1}{\eps}(\eps \lambda_2 + \eps \lambda_3) < \dfrac{1}{\eps}(2 \lambda_2 + \eps \lambda_3) < \dfrac{t}{\eps}
\end{align}
or, equivalently,
\begin{align}\label{18}
\lambda_1 > 1- \dfrac{t}{\eps}.
\end{align}
In a similar way we infer
\begin{align}
& \lambda_2' + \lambda_3' < \dfrac{t}{\eps}, \label{19}\\
& \lambda_1' > 1- \dfrac{t}{\eps}. \label{20}
\end{align}
Putting all together and taking into account that $a, a', b, b' \in K_0$ we get
$$(a,1), (a',1), (b,1), (b',1), (1-\eps)(x_0, \alpha_0) + \eps(x_1,0), (1-\eps)(y_0, \beta_0) + \eps(y_1,0) \in B_\eps,$$
so (\ref{2}), (\ref{17}), (\ref{18}), (\ref{19}), (\ref{20}) and the fact that $\diam(K_0) = 1$ allows us to conclude that
\begin{align*}
& \abs{(x, \alpha) - (y, \beta)}_\eps \leq \abs{\lambda_1 (a, 1) - \lambda_1'(a',1)}_\eps + \lambda_2 + \lambda_3 + \lambda_2' + \lambda_3' <\\
& \abs{\lambda_1 (a, 1) - \lambda_1'(a',1)}_\eps + \dfrac{2t}{\eps} \leq  \abs{\lambda_1 - \lambda_1'} \abs{(a,1)}_\eps + \lambda_1' \abs{(a,1) - (a',1)}_\eps + \dfrac{2t}{\eps}\\
& < \dfrac{t}{\eps} + \abs{(a-a', 0)}_\eps + \dfrac{2t}{\eps} \leq   \dfrac{1}{1-\eps}\norm{a-a'}_\infty + \dfrac{3t}{\eps} \leq \dfrac{1}{1-\eps}+ \dfrac{3t}{\eps}.
\end{align*}
The arbitrariness of $(x,\alpha),(y,\beta)\in S_t$ implies that $\diam_\eps(S_t) \leq \frac{1}{1-\eps} + \frac{3}{\eps}t$, as desired.
\end{proof}

Making use of Lemma~\ref{renormacion} and Proposition~\ref{infsuma} we are able to prove the following theorem.

\begin{theorem}\label{c0compl}
Let $X$ be a Banach space containing a complemented copy of $c_0$ and let $\delta \in (0,1)$. Then there exists an equivalent norm $\abs{\cdot}_\delta$ on $X$ under which $X$ has the \emph{r-BWOP} and
$$\inf\left\{ \diam_{\abs{\cdot}_\delta}(S): S \text{ \emph{slice of} } B_{\abs{\cdot}_\delta} \right\} \leq 1+\delta.$$
In particular, this happen for every separable Banach space containing $c_0$.
\end{theorem}

\begin{proof}
By assumption there exists a complemented subspace of $X$ which is isomorphic to $c_0$, so it is isomorphic to $c_0 \bigoplus_\infty \R$. Thus we can assume without loss of generality that 
$$X \simeq \left( c_0 \bigoplus\nolimits_\infty \R\right) \bigoplus\nolimits_\infty Z$$
for certain closed subspace $Z \subset X$. Now, if we consider an equivalent renorming on $Z$ such that the unit ball contains slices of arbitrarily small diameter (c.f. e.g. \cite[Lemma 2.1]{blr16}) and the equivalent norm on $c_0 \bigoplus_\infty \R$ described in Lemma~\ref{renormacion} (taking $\varepsilon=1-\frac{1}{1+\delta}$, so $1+\delta=\frac{1}{1-\varepsilon}$), we can renorm $X$ in such a way that $X = Y \bigoplus_\infty Z$ where the unit ball of $Z$ contains slices of arbitrarily small diameter and $Y$ has the r-BWOP and satisfies that the infimum of the diameter of slices of $B_Y$ is less than or equal to $1 + \delta$.

Proposition~\ref{infsuma} applies to $X$ and allows us to conclude that it has the r-BWOP, so it remains to check the second part of the theorem.

To this end observe that, given $\eps > 0$, by the condition on $Y$ we can find $f \in S_{Y^*}$ and $\alpha >0$ such that the slice
$$S_1 = \{y \in B_Y : f(y) > 1 - \alpha\}$$
satisfies
\begin{align}\label{1renorma}
\diam(S_1) < 1 + \delta + \eps.
\end{align}
On the other hand, since $B_Z$ contains slices of arbitrarily small diameter we can find $g \in S_{Z^*}$ and $\beta >0$ such that the slice
$$S_2 = \{z \in B_Z : g(z) > 1 - \beta\}$$
satisfies
\begin{align}\label{2renorma}
\diam(S_2) < 1+\delta.
\end{align}
Putting all together we can define the following (non-empty) slice of $B_X$:
$$S = \{(y,z) \in B_X: f(y) + g(z) > 2-2\gamma\},$$
where $0 < 2\gamma < \min\{\alpha, \beta\}$, and let us prove that
\begin{align}\label{3renorma}
S \subset S_1 \times S_2,
\end{align}
because in that case, using (\ref{1renorma}) together with (\ref{2renorma}), it is clear that
$$\diam(S) \leq 1+ \delta + \eps,$$
which will yield the desired consequence by the arbitrariness of $\eps>0$.

In order to check (\ref{3renorma}) observe that, given $(y, z) \in S$, then $y \in B_Y$, $z \in B_Z$ and $f(y) + g(z) > 2 - 2\gamma$. From here we get
$$f(y) > 2 -2\gamma - g(z)\geq \ 2 - 2\gamma - 1 = 1 - 2 \gamma > 1 - \alpha.$$
The fact that $g(z) > 1 - \beta$ is similar, which means $y \in S_1$ and $z \in S_2$, as desired.
\end{proof}

Now we are ready to prove Theorem~\ref{maintheorem}.

\begin{proof}[Proof of Theorem~\ref{maintheorem}]

Select any Banach space $Z$ containing a complemented copy of $c_0$. By Theorem~\ref{c0compl} we find, for every $n\in\mathbb N$, a Banach space $X_n$ isomorphic to $Z$ such that $X_n$ has the r-BWOP and
$$\inf\{\diam(S): S \text{ is a slice of } B_{X_n}\} \leq 1 + \dfrac{1}{n}.$$
Define $X= \ell_1(\N, X_n)$ so, by Proposition~\ref{psuma}, we get that $X$ has the r-BWOP. To finish, let us prove that given $0 <\eps <1$ there exists a slice $S_\eps$ of $B_X$ whose diameter is less than or equal to $1+\eps$.

In order to do so, take $n > 4/\eps$ and select $S_n$ to be a slice of $B_{X_n}$, defined by $f_n \in S_{X_n^*}$ and $\alpha > 0$, such that $\diam(S_n) \leq 1+ 2/n$. Take $0 < \delta < \min\{\alpha, \eps/4\}$ and define the (non-empty) slice
$$S_n' = \{x \in B_{X_n}: f_n(x) > 1- \delta\},$$
which clearly satisfies $S_n' \subset S_n$ and thus $\diam(S_n') \leq 1+2/n$. Finally define
$$S_\eps = \{x \in B_X: (f_n \circ \pi_n)(x) > 1-  \delta\},$$
and let us prove that $S_\varepsilon$ is the desired slice. Indeed, given $x=(x_m), y=(y_m) \in S_\eps$, it is clear that $x_n, y_n \in S_n'$, so
\begin{align}\label{remate1}
\norm{x_n - y_n} \leq 1+\dfrac{2}{n} < 1+\dfrac{\eps}{2}.
\end{align}
Moreover, taking into account that $\norm{x_n} \geq f_n(x_n) > 1-\delta$ we obtain
\begin{align}\label{remate2}
\sum_{\substack{m=1 \\ m \ne n}}^\infty \norm{x_m} = \norm{x}_1 - \norm{x_n} \leq 1 - (1-\delta) < \dfrac{\eps}{4}.
\end{align}
Similarly, we get
\begin{align}\label{remate3}
\sum_{\substack{m=1 \\ m \ne n}}^\infty \norm{y_m} < \dfrac{\eps}{4}.
\end{align}
Combining (\ref{remate1}), (\ref{remate2}) and (\ref{remate3}) we get
\begin{align*}
\norm{x-y}_1 & = \sum_{m=1}^\infty \norm{x_m - y_m}  = \norm{x_n-y_n} + \sum_{\substack{m=1 \\ m \ne n}}^\infty \norm{x_m - y_m} \\
& \leq \norm{x_n - y_n} + \sum_{\substack{m=1 \\ m \ne n}}^\infty \norm{x_m} + \sum_{\substack{m=1 \\ m \ne n}}^\infty \norm{y_m} < 1 + \eps
\end{align*}
which proves that $\diam(S_\varepsilon)\leq 1+\varepsilon$ by the arbitrariness of $x,y\in S_\varepsilon$.
\end{proof}

Several remarks are pertinent.

\begin{remark}\label{remark:conseqmaintheor}
\begin{enumerate}
\item An inspection in the proof of Theorem~\ref{maintheorem} reveals that, indeed, given any Banach space $X$ containing a complemented copy of $c_0$ (in particular, any separable Banach space $X$ containing $c_0$), there exists a sequence of Banach spaces $X_n$ isomorphic to $X$ such that $\ell_1(\N, X_n)$ satisfies the conclusion of the theorem. This reveals a wide class of non-isomorphic Banach spaces satisfying the conclusion of the theorem. This should be compared with the examples constructed in \cite[Theorem 3.1]{rz23} of Banach spaces with the r-BSP and such that the infimum of the diameter of slices is $1$, which are based on very specific tree spaces. 

\item At this point it is natural to wonder whether the above mentioned spaces constructed in \cite[Theorem 3.1]{rz23} could have been used to distinguish the r-BWOP and the D2P. In fact, the answer is no, because all the spaces constructed there satisfy the CPCP, in particular, the unit ball of all these spaces contains non-empty relatively weakly open subsets with arbitrarily small radius.

\item As we have pointed out before, in \cite{rz23} Banach spaces with the r-BSP and the CPCP are constructed. A natural question, whose answer we do not know, is whether there are Banach spaces with the r-BWOP and being strongly regular.
\end{enumerate}
\end{remark}

\section*{Acknowledgements}

This research has been supported  by MCIU/AEI/FEDER/UE\\  Grant PID2021-122126NB-C31, by MICINN (Spain) Grant \\ CEX2020-001105-M (MCIU, AEI) and by Junta de Andaluc\'{\i}a Grant FQM-0185. The research of E. Mart\'inez Va\~n\'o has also been supported by Grant PRE2022-101438 funded by MCIN/AEI/10.13039/501100011033 and ESF+. The research of A. Rueda Zoca was also supported by Fundaci\'on S\'eneca: ACyT Regi\'on de Murcia grant 21955/PI/22.

\end{document}